\documentclass[12pt]{article}

\usepackage{amsmath,amsfonts,amssymb}
\usepackage{amsthm}
\usepackage{cite}
\usepackage{multirow}
\usepackage{graphicx}

\newtheorem{Def}{Definition}
\newtheorem{Rem}{Remark}
\newtheorem{Thm}{Theorem}
\newtheorem{Lem}{Lemma}

\title{A Winning Strategy for the Game of Antonim}
\author{Zachary Silbernick \\ Robert Campbell}

\begin{document}
	
\maketitle

\begin{abstract} The game of Antonim is a variant of the game Nim, with the additional rule that heaps are not allowed to be the same size. A winning strategy for three heap Antonim has been solved. We will discuss the solution to three-heap Antonim and generalize this theory to an arbitrary number of heaps. \end{abstract}

\section{Introduction}
	
Nim is a classic game between two players with a certain amount of heaps and varied amount of chips inside each heap. The rules of the game are that players take turns taking as many chips out of a single heap. The player who can no longer take any chips (i.e. there are no chips remaining) loses the game. Nim can also be played with coins on a non-negative number line. The number of coins corresponds to the number of heaps, and the space on the number line where each coin lays corresponds to the number of chips in the heap. For example, if a coin is on the space numbered 9, that would represent nine chips in a heap. When Nim is played on a number line, players take turns moving coins to the left until all coins have been moved to the 0 space. The player that moves the last coin to the 0 space is the winner. A winning strategy for the game of Nim is known \cite{1901}.  
	
	\subsection{Antonim}
	
Antonim (also known as Antipathetic Nim) is a variation of Nim in which no two heaps are allowed to have the same number of chips.  A solution for three heap Antonim is known, however a general solution for Antonim was previously unknown \cite{MR2006327} \cite{MR1973110}. In this paper we will examine several game theory theorems and definitions as well as the solution for three heap Antonim. With this knowledge in mind, we will then discuss a general solution for the game of Antonim.
\section{Game Theory} 

\begin{Def}A \underline{game-state} is a position of the pieces that is achievable under the rules of a game. \end{Def} 

We will represent a game-state in Antonim as an ordered tuple ($x_1,x_2,\ldots, x_n$) where $n$ represents the number of heaps and $x_i \in \mathbb{Z} ^{\geq 0}$ represents the number of chips in a heap. Note, by the rules of Antonim, $x_i=x_j$ iff $i=j$ 

\begin{Def} A \underline{$\mathcal{P}$-position} is defined as a game-state where whoevers turn it is will lose assuming the other player plays perfectly.\end{Def}

\begin{Def}An \underline{$\mathcal{N}$-position} is defined as all other game-states that are not $\mathcal{P}$-positions. \end{Def} 

\begin{Lem} Every possible move from a $\mathcal{P}$-position is to an $\mathcal{N}$-position\cite{MR1808891}. \end{Lem} 
\begin{Lem} For every $\mathcal{N}$-position there exists at least one move to a $\mathcal{P}$-position \cite{MR1808891}. \end{Lem} 

\section{Winning Strategy}

A winning strategy for Antonim would be to always give our opponent $\mathcal{P}$-positions. If we give our opponent a $\mathcal{P}$-position, then since it is our opponent's turn, and the game state is a $\mathcal{P}$-position, by definition, our opponent will lose the game. Note from a $\mathcal{P}$-position, our opponent can only give us $\mathcal{N}$-positions (Lemma 1). From any $\mathcal{N}$-position given to us, we can give a $\mathcal{P}$-position back to our opponent (Lemma 2). We can continue to follow this process until our opponent can no longer take any chips. It follows that we would win the game.  

\begin{Rem} Finding a winning strategy for Antonim is equivalent to finding the $\mathcal{P}$-positions for any game of Antonim. \end{Rem}

 \section{The existence of $\mathcal{P}$-positions}

\begin{Thm} Let $x_1,x_2, \ldots, x_{n-1}\in \mathbb{Z} ^{\geq 0}$ such that $x_i=x_j$ iff $i=j$. There exists $z\in \mathbb{Z} ^{\geq 0}$ such that $\left(x_1, x_2, \ldots, x_{n-1},z\right )$ is a $\mathcal{P}$-position for Antonim. \end{Thm}

\begin{proof} Suppose, for a proof by contradiction, that $\exists x_1, x_2, \ldots, x_{n-1} \in \mathbb{Z} ^{\geq 0}$ and the position ($x_1,x_2 \ldots, x_{n-1},z$) is an $\mathcal{N}$-position for all $z \in \mathbb{Z} ^{\geq 0}$ where $z\neq x_i$ for $1\leq i \leq n-1$. There are at most $x_i$ moves that can be made from the $i^{th}$ position and thus there are at most $x_1\times x_2 \times \cdots \times x_{n-1}$ moves that can be made from the first $n-1$ positions of the game state $\left(x_1, x_2, \ldots, x_{n-1},z\right )$. Fix $z_0$ s.t. $z_0 \neq x_i$. One of the possible moves for the game state $\left(x_1, x_2, \ldots, x_{n-1},z_0\right )$ has to be to a $\mathcal{P}$-position since $\left(x_1, x_2, \ldots, x_{n-1},z_0\right )$ is an $\mathcal{N}$-position (Lemma 2). WLOG say that $\left(x_1-y_1, x_2, \ldots, x_{n-1}, z_0\right)$ is a $\mathcal{P}$-position (where $x_1-y_1 \in \mathbb{Z} ^{\geq 0}$ and $x_1-y_1 \neq x_i)$. This $\mathcal{P}$-position is unique to $z_0$ since for all other valid game states ($x_1, x_2, $ $\ldots, x_{n-1},z_0+g$) (where $g\in \mathbb Z^{\geq 0}$) the game state $\left(x_1-y_1, x_2, \ldots, x_{n-1},z_0+g\right)$ is not a $\mathcal{P}$-position. Indeed, from $\left(x_1-y_1, x_2, \ldots, x_{n-1},z_0+g\right)$ one could move to the $\mathcal{P}$-position $\left(x_1-y_1, x_2, \ldots, x_{n-1}, z_0\right)$ by taking $g$ chips from the $n$th heap. Thus each value $x_1,x_2,\ldots,x_{n-1}$ that makes ($x_1,x_2,\ldots,x_{n-1},z$) a $\mathcal{P}$-position is unique to $z$. There are a finite amount of moves from ($x_1, x_2, \ldots, x_{n-1},z$), but an infinite amount of $z$s, each of which has a single move from ($x_1, x_2, \ldots, x_{n-1},z$) to a $\mathcal{P}$-position. We have a contradiction. Therefore, for each valid ($x_1,x_2,\ldots,x_{n-1}$) there exists $z\in \mathbb{Z} ^{\geq 0}$ such that ($x_1, x_2, \ldots, x_{n-1},z$) is a $\mathcal{P}$-position.  \end{proof}

\section{$\mathcal{P}$-positions for three heap Antonim}
The following table and theorem can be found in \underline{Winning Ways} \cite{MR2006327}. An extended version of this table would give the $\mathcal{P}$-positions for any 3-heap game of Antonim. Let us call this the $3 \mathcal{P}$-Antonim Table (Table \ref{3p}).  The $3 \mathcal{P}$-Antonim Table gives the value $z$ that makes the Antonim game ($x_1,x_2,z$) a $\mathcal{P}$-position.  
\begin{table} \begin{tabular}{c|cccccccccccccc}

&0&1&2&3&4&5&6&7&8&9&10&11&12\\  \hline
0&0&2&1&4&3&6&5&8&7&10&9&12&11\\
1&2&X&0&5&6&3&4&9&10&7&8&13&14\\
2&1&0&X&6&5&4&3&10&9&8&7&14&13\\
3&4&5&6&X&0&1&2&11&12&13&14&7&8\\
4&3&6&5&0&X&2&1&12&11&14&13&8&7\\
5&6&3&4&1&2&X&0&13&14&11&12&9&10\\
6&5&4&3&2&1&0&X&14&13&12&11&10&9\\
7&8&9&10&11&12&13&14&X&0&1&2&3&4\\
8&7&10&9&12&11&14&13&0&X&2&1&4&3\\
9&10&7&8&13&14&11&12&1&2&X&0&5&6\\
10&9&8&7&14&13&12&11&2&1&0&X&6&5\\
11&12&13&14&7&8&9&10&3&4&5&6&X&0\\
12&11&14&13&8&7&10&9&4&3&6&5&0&X\\
13&14&11&12&9&10&7&8&5&6&3&4&1&2&\\
14&13&12&11&10&9&8&7&6&5&4&3&2&1\\
\end{tabular} \caption{Part of the $3 \mathcal{P}$ Antonim Table}\label{3p} \end{table}  

Table \ref{3p} was created with the following theorem. 

\begin{Thm}
The game state ($x_1,x_2,z$) is a $\mathcal{P}$-position in Antonim if and only if ($x_1+1,x_2+1,z+1)$ is a $\mathcal{P}$-position in Nim \cite{MR2006327}. \end{Thm} 

Unfortunately, this theorem does not hold for games of Antonim with more than three heaps. For example, the game state $(1,2,5,6)$ is a $\mathcal{P}$-position in Nim, but $(0,1,4,5)$ is not a $\mathcal{P}$-position for Antonim. Notice $(0,1,4,5)$ is not listed in the $3 \mathcal{P}$ Antonim Table (Table \ref{3p}), and from $(0,1,4,5)$ one could give the $\mathcal{P}$-position $(0,1,3,5)$. Since one can move to a $\mathcal{P}$-position from $(0,1,4,5)$ by Lemma 2, the game-state $(0,1,4,5)$ must be a $\mathcal{N}$-position.

\section{Higher Dimensional Antonim}
	
	There is a pattern in table 1. An entry $z$ that makes the Antonim game ($x_1,x_2,z$) a $\mathcal{P}$-position is filled in as the least positive integer not coinciding with any earlier entry in the same row or column, nor coinciding with either the row or column heading of the table \cite{MR2006327}. \
	
	For illustration, if we were to start filling out the $3 \mathcal{P}$-Antonim Table above using this pattern, we would start in the top left corner. Clearly, the game-state ($0,0,0$) is a $\mathcal{P}$-position, but what value of $z$ makes ($0,1,z$) a $\mathcal{P}$-position? If we follow the pattern we noticed above, $z$ cannot be 1 because 1 is already listed in column heading. It cannot be 0 as well because 0 is already listed in the row. After eliminating these values, the least positive integer left over would be 2. Thus ($0,1,2$) is a $\mathcal{P}$-position. Continuing down the first row of Table 1, ($0,1,2$) is again a $\mathcal{P}$-position, but what value of $z$ makes the Antonim game ($0,3,z$) a $\mathcal{P}$-position? Well, 3 cannot be the value of $z$ because it is listed in the column heading, and 0, 1, and 2 cannot be the value of $z$ because they are previous values in the row. The least positive integer left is 4. Thus ($0,3,4$) is a $\mathcal{P}$-position. We can continue to use this process to fill out the rest of the $3 \mathcal{P}$-Antonim Table. Throughout this example, notice that we never eliminated 0 when it appeared in the row or column heading. The value 0, is a special case. We eliminate 0, as the value of $z$ when it appears in the row or column, but when it appears in a row or column heading, we do not eliminate it as a possible value for $z$ as 0 could make the game-state a $\mathcal{P}$-position. Case and point, the value of $z$ for the game-state $(0,0,z)$ is in fact 0 even though 0 appears in the row and column heading. 

	The pattern in table 1 does not extend to Antonim with more heaps given its present wording. We can, however, generalize this idea as follows: We will now call the rows of a table of Antonim dimension 1 (denoted $d_1$) and the columns of a table of Antonim dimension 2 ($d_2$). In games of Antonim with greater than three heaps, new dimensions come into play. For example, in four heap Antonim the new dimension that presents itself will be dimension 3, ($d_3$).
	
	An example with this new dimension, ($d_3$), is shown in the $4 \mathcal{P}$-Antonim table below (table \ref{4p}). Since we can't reproduce this in one table we make the $4 \mathcal{P}$-Antonim table with layers of tables. The first table is the first layer of the $4 \mathcal{P}$-Antonim table. The bold zero in the top left corner represents 0 chips in the first heap of a four heap game of Antonim. (This is three heap Antonim).

\begin{table}	
\begin{tabular}{c|cccccc} \LARGE $\mathbf{0}$

 &0&1&2&3&4&5\\  \hline
0&0&2&1&4&3&6\\
1&2&X&0&5&6&3\\
2&1&0&X&6&5&4\\
3&4&5&6&X&0&1\\
4&3&6&5&0&X&2\\
5&6&3&4&1&2&X\\

\end{tabular} \ \ \ \ \ \ \ \
\begin{tabular}{c|cccccc} \LARGE $\mathbf{1}$

 &0&1&2&3&4&5\\  \hline
0&2&X&0&5&6&3\\
1&X&X&X&X&X&X\\
2&0&X&X&4&3&6\\
3&5&X&4&X&2&0\\
4&6&X&3&2&X&7\\
5&3&X&6&0&7&X\\

\end{tabular}  
\caption{Part of the $4 \mathcal{P}$-Antonim Table}\label{4p} \end{table}  

Let $\left(x_1, x_2, \ldots, x_{n-1}, z\right )$ be a $\mathcal{P}$-position in a game of Antonim where $x_i \in Z^{> 0}$. Let $$A=\left\{\alpha \in \mathbb Z^{\geq 0} | \exists y_j \leq x_j \:where \:(x_1,\ldots, x_{j-1},x_j-y_j,x_{j+1},\ldots, x_{n-1},\alpha) \:is\: a\: \mathcal{P}-position\right\}$$ (the set of all values that already appear in the other dimensions of the $n \mathcal{P}$-Antonim table).

\begin{Lem} $z\notin A\cup \left\{x_1, \ldots, x_{n-1}\right\}$. \end{Lem}

\begin{proof}  We know $z\neq x_1, x_2, \ldots, x_{n-1}$, by the rules of Antonim. Let $\alpha \in A$. Then there exists $y_j \leq x_j$ such that $(x_1,\ldots, x_{j-1},x_j-y_j,x_{j+1},\ldots, x_{n-1},\alpha)$ is a $\mathcal{P}$-position. From  $\left(x_1, x_2, \ldots, x_{n-1}, \alpha \right )$ there is a move to $(x_1,\ldots, x_{j-1},x_j-y_j,x_{j+1},\ldots, x_{n-1},\alpha)$. Since $(x_1,\ldots, x_{j-1},x_j-y_j,x_{j+1},\ldots, x_{n-1},\alpha)$ is a $\mathcal{P}$-position, $\left(x_1, x_2, \ldots, x_{n-1}, \alpha \right )$ must be a $\mathcal{N}$-position by Lemma 1. Therefore, $z\notin A$ \end{proof}

With Lemma 3, we know what values $z$ cannot be, namely, any of the earlier values in the dimensions or headings of an Antonim table. Our next theorem will show that $z$ is in fact the least non-negative integer left over after all of the previous values have been eliminated. 

\section{Main Theorem}

\begin{Thm} Let $\left(x_1, x_2, \ldots, x_{n-1}\right )$ be an ($n-1$)-heap game of Antonim where $x_i \in Z^{> 0}$. Let $A$ be as before. Let $ S= \mathbb{Z} ^{\geq 0} - (A\cup \left\{x_1, \ldots, x_{n-1}\right\}$). Let $z$ be the least element in $S$. Then $\left(x_1, x_2, \ldots, x_{n-1}, z\right )$ is a $\mathcal{P}$-position. \end{Thm}

\begin{proof} In order to show $\left(x_1, x_2, \ldots, x_{n-1}, z\right )$ is a $\mathcal{P}$-position, we want to show all possible moves in the game-state $\left(x_1, x_2, \ldots, x_{n-1}, z\right )$ are to $\mathcal{N}$-positions. Notice $\exists$ $\alpha \in A$ such that $\left(x_1-y_1, x_2, \ldots, x_{n-1}, \alpha \right)$ is a $\mathcal{P}$-position, and since $z\neq \alpha$,  $\left(x_1-y_1, x_2, \ldots, x_{n-1}, z \right)$ is a $\mathcal{N}$-position since from $\left(x_1-y_1, x_2, \ldots, x_{n-1}, z \right)$ one could move to the $\mathcal{P}$-position $\left(x_1-y_1, x_2, \ldots, x_{n-1}, \alpha \right)$. Similar arguments show, $\left(x_1, x_2-y_2, \ldots, x_{n-1}, z\right)$, \\ $\left(x_1, x_2, x_3-y_3 \ldots, x_{n-1}, z\right)$, \ldots,$\left(x_1, x_2, \ldots, x_{n-1}-y_{m-1}, z\right)$ are all $\mathcal{N}$-positions. Also, \\ $\left(x_1, x_2, \ldots, x_{n-1}, z-y\right )$ is a $\mathcal{N}$-position (for any value $0\leq y \leq z$) because $z-y\in (A\cup \left\{x_1, \ldots, x_{n-1}\right\})$, which by definition is an $\mathcal{N}$-position or an invalid game-state in Antonim. Since all possible moves of $\left(x_1, x_2, \ldots, x_{n-1}, z\right )$ are to $\mathcal{N}$-positions, by Lemma 1, $\left(x_1, x_2, \ldots, x_{n-1}, z\right )$ is a $\mathcal{P}$-position. \end{proof}

\section{Conclusion}

With thereom 3, we know the $\mathcal{P}$-positions for any game-state of Antonim. Since we can now find a $\mathcal{P}$-position for any n-heap game of Antonim, we can use these $\mathcal{P}$-positions to always win a game of Antonim. With this winning strategy, there is not a lot of research that could be continued on understanding Antonim, however, one could apply this research to other fields of mathematics where games are involved.

\bibliography{ultimatebibliography}{}
\bibliographystyle{plain}

\end{document}